\documentclass[sn-mathphys,pdflatex]{sn-jnl}

\usepackage{amssymb}
\usepackage{amsmath}
\usepackage{amsthm}
\usepackage{stmaryrd}
\usepackage{mathtools}
\usepackage{booktabs}
\usepackage{enumerate}

\catcode`\|=12\relax
\renewcommand\S{§}

\usepackage{tikz}
\usepackage{url}
\usepackage{xfrac}

\newcommand{\dcup}{\mathbin{\dot\cup}}

\newcommand{\gauss}[2]{\genfrac{[}{]}{0pt}{}{#1}{#2}}
\newcommand{\eps}{\varepsilon}

\renewcommand{\le}{\leqslant}
\renewcommand{\ge}{\geqslant}
\renewcommand{\leq}{\leqslant}
\renewcommand{\geq}{\geqslant}
\newcommand{\coloneq}{\vcentcolon=}      
\newcommand{\lhdeq}{\trianglelefteqslant}    

\newcommand{\rA}{{\mathrm{A}}}
\newcommand{\GL}{{\mathrm{GL}}}
\newcommand{\GO}{{\mathrm{GO}}}
\newcommand{\GU}{{\mathrm{GU}}}
\newcommand{\Sp}{{\mathrm{Sp}}}
\newcommand{\Sym}{{\mathrm{Sym}}}
\newcommand{\F}{{\mathbb F}}
\newcommand{\qK}{q{\kern-1pt}K}

\jyear{2023}%

\theoremstyle{thmstyleone}%
\newtheorem{theorem}{Theorem}[section]

\newtheorem{Theorem}[theorem]{Theorem}
\newtheorem{Proposition}[theorem]{Proposition}
\newtheorem{Lemma}[theorem]{Lemma} 
\newtheorem{Corollary}[theorem]{Corollary}
\theoremstyle{thmstyletwo}%

\theoremstyle{thmstylethree}%

\begin{document}

\title[The proportion of non-degenerate complementary subspaces]{The proportion of non-degenerate complementary subspaces
  in classical spaces}


\author[1]{\fnm{S.P.} \sur{Glasby}}\email{Stephen.Glasby@uwa.edu.au}

\author[2]{\fnm{Ferdinand} \sur{Ihringer}}\email{Ferdinand.Ihringer@ugent.be}

\author[3,4]{\fnm{Sam} \sur{Mattheus}}\email{SMattheus@ucsd.edu}

\affil[1]{\orgdiv{Center for the Mathematics and Symmetry and Computation}, \orgname{University of
  Western Australia}, \orgaddress{\city{Perth}, \postcode{6009}, \country{Australia}}}

\affil[2]{\orgdiv{Department~of Mathematics:~Analysis, Logic and Discrete
Mathematics}, \orgname{Ghent University}, \orgaddress{\country{Belgium}}}

\affil[3]{\orgdiv{Department~of Mathematics and Data Science}, \orgname{Vrije Universiteit Brussel}, \orgaddress{\country{Belgium}}}

\affil[4]{\orgdiv{Department~of Mathematics}, \orgname{University of 
California San Diego}, \orgaddress{\country{United States}}}


\abstract{
  Given positive integers $e_1,e_2$, let $X_i$ denote the set of
  $e_i$-dimensional subspaces of a fixed finite vector space $V=(\F_q)^{e_1+e_2}$.
  Let $Y_i$ be a non-empty subset of $X_i$ and let $\alpha_i = |Y_i|/|X_i|$.
  We give a positive lower bound, depending only on $\alpha_1,\alpha_2,e_1,e_2,q$,
  for the proportion of pairs
  $(S_1,S_2)\in Y_1\times Y_2$ which intersect trivially.
 As an application, we bound the proportion of pairs of non-degenerate
 subspaces of complementary dimensions in a finite classical space
 that intersect trivially. This problem is motivated by an algorithm for
 recognizing classical groups. By using techniques from algebraic graph theory,
 we are able to handle orthogonal groups over the field of order~2, a case
 which had eluded Niemeyer, Praeger, and the first author.
}




\maketitle

\section{Introduction}

  In this paper we use techniques from algebraic graph theory to solve a
  problem that arose from computational group theory. More precisely, we
  use the expander mixing lemma for bipartite graphs
  to establish bounds that are useful
  for algorithms which `recognise' classical groups acting on their
  natural module, a central and difficult computational problem. The nature of
  this recognition problem is sketched in Section~\ref{ss}.

  Let $\F$ be a finite field, let $e_1,e_2$ be positive integers
  and let $V=\F^{e_1+e_2}$ be an $(e_1+e_2)$-dimensional
  $\F$-space endowed with a non-degenerate quadratic, symplectic or hermitian
  form. We bound the probability that a non-degenerate $e_1$-subspace $S_1$
  of $V$,
  and a non-degenerate $e_2$-subspace $S_2$ of $V$, intersect trivially
  that is, satisfy $S_1\cap S_2=\{0\}$.
  Except for orthogonal spaces with $q=2$,
  this problem was solved in~\cite[Theorem~1.1]{GNP2022},
  using a combinatorial double-counting argument~\cite[\S3]{GNP2022}.
  The following theorem
  gives sharper bounds, without exception, and is proved via relatively
  straightforward
  calculations involving the second largest eigenvalue of a graph,
  see Section~\ref{ss:k}.
  
  \begin{Theorem}\label{T:OSpU}
    Let $V=\F^{e_1+e_2}$ be a non-degenerate orthogonal, symplectic or hermitian
    space where $\F,e_1,e_2$ are given in Table~\ref{T:data}. Let $Y_i$
    be the set of all non-degenerate $e_i$-subspaces of $V$ (of a fixed type
    $\sigma_i\in\{-,+\}$ in the orthogonal case). Then the proportion of
    pairs $(S_1,S_2)\in Y_1\times Y_2$ for which $S_1\cap S_2=\{0\}$ is
    at least $1-\frac{c}{|\F|}$ where $c$ is given in Table~\ref{T:data}.
    We may take $c=\frac{3}{2}$ if $(e_1,e_2,q)\ne(1,1,2)$.
  \end{Theorem}

  \begin{table}[!ht]
    \label{T:data}
    \caption{Choices for $\F$, $e_1$, $e_2$, form on $V$, $Y_i$ and $c$
    in Theorem~\ref{T:OSpU}}\vskip3mm
    \centering
  \begin{tabular}{lcclll}
    \toprule
    $\F$& $e_1$ & $e_2$& form on $V=\F^{e_1+e_2}$ & $e_i$-subspaces $Y_i$&$c$\\ 
    \midrule
    $\F_q$& even & even & orthogonal of & non-degenerate of&$\sfrac{3}{2}$\\ 
     &  &  & type $\eps\in\{-,+\}$& type $\sigma_i\in\{-,+\}$&\\ 
    $\F_q$& even & even & symplectic& non-degenerate&$\sfrac{10}{7}<\sfrac{3}{2}$\\ 
    $\F_{q^2}$& $\ge1$ & $\ge1$ & hermitian& non-degenerate&  $\sfrac{3}{2}$\\
    &  &  & &$(e_1,e_2,q)\ne(1,1,2)$& \\
    \bottomrule
  \end{tabular}
  \end{table}

  The non-degenerate subspaces of a symplectic space have even dimension
  so that $e_1$ and $e_2$ are both even in the second line of
  Table~\ref{T:data}. In the case that $V$ is a non-degenerate orthogonal space
  the $e_i$ are also both even; however this is for a different reason.
  The authors of~\cite{GNP2022a}, in forthcoming work, describe an algorithm
  for recognising classical groups, and the papers~\cite{GNP2022,GNP2022a}
  provide the necessary background. In the algorithmic application the subspace
  $S_i$ is the image of $g_i-1$ for some element $g_i$ of the orthogonal
  group on $V$, and $S_i$ is non-degenerate of \emph{minus} type
  by~\cite[Lemma~3.8(b)]{GNP2022a}. Hence the $e_i$ are even in the first
  line of Table~\ref{T:data}, as claimed.

  In the unitary case, we have $c=2$ when $(e_1,e_2,q)=(1,1,2)$ and
  $c=1.26$ when $e_1,e_2\ge2$, see Theorem~\ref{T:U}.
  The bounds listed in Table~\ref{T:data} all
  satisfy  $1-\frac{3}{2|\F|}\ge\frac{1}{4}$, and they
  facilitate a uniform analysis, for all fields,
  of a randomized algorithm for recognising classical
  groups.  In contrast, the values of $c$ in~\cite[Table~1]{GNP2022} are
  2.69 (for $q\ge3$), 1.67 and 1.8 in the
  orthogonal, symplectic and unitary cases, respectively.
  Further, our methods are somewhat stronger and easier to apply than those
  in~\cite{GNP2022}, and offer hope for extensions, see Section~\ref{S:future}.

\subsection{\texorpdfstring{$q$}{q}-Kneser graphs}\label{ss:k}
Let $V=(\F_q)^d$ be a $d$-dimensional vector space over the field
with $q$ elements. Let $e_1,e_2$ be positive integers. 
For $i=1,2$, denote by $X_i$ the set of $e_i$-dimensional subspaces of $V$.
We refer to an element $S_i\in X_i$ as an \emph{$e_i$-subspace}
or an \emph{$e_i$-space}.
Let $\Gamma_{d,e_1,e_2}$ be the bipartite graph whose vertex set is the
disjoint union $X_1 \dcup X_2$ (where we take two disjoint copies 
of the set of $e_1$-spaces if $e_1=e_2$), and 
where two vertices $(S_1, S_2) \in X_1 \times X_2$ are adjacent whenever
$S_1$ and $S_2$ intersect trivially. The condition $S_1 \cap S_2 =\{0\}$ is equivalently to
$\dim(S_1+ S_2)=e_1+e_2$.

\smallskip

The $q$-Kneser graph $\qK(d,e)$ has been previously studied,
for example, 
see \cite{BBS2011,CCESW2020}. The vertices of $\qK(d,e)$
comprise $e$-subspaces
of $V=(\F_q)^d$ and $\{S_1,S_2\}$ is an edge if $S_1\cap S_2=\{0\}$.
If $\qK(d,e)$ has adjacency matrix $A$, then the \emph{bipartite double}
of $\qK(d,e)$ has adjacency matrix
$\left[\begin{smallmatrix}0&A\\A^T&0\end{smallmatrix}\right]$ and is isomorphic
to $\Gamma_{d,e,e}$. 
The spectrum of $qK(d,e)$ (i.e, the set of eigenvalues of $A$) is known,
and hence too for its bipartite double $\Gamma_{d,e,e}$, and can be obtained from
Delsarte \cite[Theorem~10]{Delsarte1976} or Eisfeld \cite[Theorem~2.7]{Eisfeld1999}.
The spectrum of $\Gamma_{d,e_1,e_2}$ when $e_1\ne e_2$ is more complicated. 
Brouwer \cite{Brouwer2010} gives the spectrum of $\Gamma_{3,1,2}$; also
Suda and Tanaka~\cite{ST2014} study ``cross-independent'' sets
in $\Gamma_{d,e_1,e_2}$ with $d \geq 2e_1, 2e_2$. However, for our applications
we want $d=e_1+e_2$, as this is the key case for~\cite{GNP2022}
which underpins~\cite{GNP2022a}.

\smallskip 

We henceforth assume that $d=e_1+e_2$, and write
$\Gamma_{e_1,e_2} = \Gamma_{d,e_1,e_2}$.
Since $\Gamma_{e_1,e_2}\cong\Gamma_{e_2,e_1}$,
we shall assume additionally, without loss of generality, that $e_1\ge e_2$.

\smallskip 

For each $e_1$-subspace $S_1$ of $(\F_q)^{e_1+e_2}$, there are $q^{e_1e_2}$
choices for an $e_2$-subspace $S_2$ with $S_1\cap S_2=\{0\}$. Similarly,
for each $e_2$-subspace $S_2$ there are $q^{e_2e_1}$
choices for an $e_1$-subspace $S_1$ with $S_1\cap S_2=\{0\}$.
Hence the graph $\Gamma_{e_1,e_2}$ is $q^{e_1e_2}$-regular.
The following result is proved in Section~\ref{sec:eigenvalues}, it determines
the distinct eigenvalues of $\Gamma_{e_1,e_2}$, but not their multiplicities.

\begin{Proposition}\label{lem:ev}
  Suppose that $e_1\ge e_2 \ge 1$ and $d=e_1+e_2$.
  The distinct eigenvalues of the bipartite graph $\Gamma_{e_1,e_2}$ are
  $\lambda_0>\cdots>\lambda_{e_2}>-\lambda_{e_2}>\cdots>-\lambda_0$
  where $\lambda_j=q^{m_j}$ for $0 \leq j \leq e_2$
  and $m_j = e_1e_2 - \frac{j(d+1-j)}{2}$.
\end{Proposition}

The number $\gauss{a}{b}_q$ of $b$-subspaces of the
$a$-dimensional vector space $(\F_q)^a$ is
\[
\gauss{a}{b}_q = \prod_{i=0}^{b-1} \frac{q^{a-i}-1}{q^{b-i}-1}
=\prod_{i=0}^{b-1}\frac{q^{a-i-1}+\cdots+q+1}{q^{b-i-1}+\cdots+q+1}
= \prod_{i=1}^b \frac{q^{a-i+1}-1}{q^i-1}.
\]
The second middle product shows that
$\lim_{q\to 1}\gauss{a}{b}_q = \prod_{i=0}^{b-1} \frac{a-i}{b-i}=\binom{a}{b}$, and
$\gauss{a}{b}_q\sim q^{b(a-b)}$ as $q\to\infty$.
The next result is proved in Section~\ref{S3} using
Proposition~\ref{lem:ev}, and the
expander mixing lemma for regular bipartite graphs, see Lemma~\ref{lem:eml}.

\begin{Proposition}\label{lem:density}
  Suppose that $e_1\ge e_2 \ge 1$ and $d=e_1+e_2$.
  Let $Y_1 \subseteq X_1$ and $Y_2 \subseteq X_2$ be non-empty.
  Put $\alpha_i = |Y_i|/|X_i|$ for $i \in \{ 1, 2\}$.
  Then $\alpha_1\alpha_2>0$~and
  \smallskip
  \[
    \frac{|\{(S_1, S_2) \in Y_1 \times Y_2: S_1 \cap S_2 = 0\}|}{|Y_1| \cdot |Y_2|}  
    \geq \frac{q^{e_1e_2}}{\gauss{d}{e_1}_q}
    \left(1 - \sqrt{(\tfrac{1}{\alpha_1}-1)(\tfrac{1}{\alpha_2}-1)} q^{-\frac{d}{2}}\right).
  \]
  Suppose that $\min\{\alpha_1,\alpha_2\}\ge\alpha>0$
  and $\omega_q(e)=\prod_{i=1}^{e}(1-q^{-i})$. Then
  \smallskip
  \[
    \frac{|\{(S_1, S_2) \in Y_1 \times Y_2: S_1 \cap S_2 = 0\}|}{|Y_1| \cdot |Y_2|}  
    >\omega_q(e_2)\left(1 -\left(\frac{1}{\alpha}-1\right)q^{-\frac{d}{2}}\right).
  \]
\end{Proposition}



\subsection{Recognising classical groups and outline of the paper}\label{ss}
A group $G$ satisfying $\textup{SL}_d(q)\lhdeq G\le\GL_d(q)$
is generated by a set $\mathcal{X}$ of elementary matrices,
corresponding to elementary row operations. Furthermore, given an
element $g\in G$ there is an efficient algorithm (e.g.\ Gaussian elimination)
which writes $g$ as a word in $\mathcal{X}$. However, in computational problems
$G$ may be generated by a set $\mathcal{Y}$ of arbitrary-looking matrices, and
`recognising' $G$ involves writing each element of $\mathcal{X}$
as a word in $\mathcal{Y}$. A particularly helpful special case
is when $G$ is generated by a set $\mathcal{Y}'=\{g_1,g_2\}$ of two matrices
with $S_1=\textup{im}(g_1-1)$ and $S_2=\textup{im}(g_2-1)$ non-degenerate complementary subspaces, and a key problem is to write each element of $\mathcal{X}$
as a word in $\mathcal{Y}'$. This problem is practically difficult,
as is the analogous problem for classical groups,  and its solution uses
random selections in $G$ and the natural $G$-module $V=(\F_q)^d$. The authors
of~\cite{GNP2022,GNP2022a} describe in forthcoming work an algorithm 
to solve this word problem, and the translation from a problem in group theory to a
geometric problem is described, in part, in~\cite{GNP2022,GNP2022a}.
Further context and details are given in~\cite[Section~3]{GNP2022a}.

In Section~\ref{sec:eigenvalues}, we determine the distinct eigenvalues of $\Gamma_{e_1,e_2}$ by proving
Proposition~\ref{lem:ev}. The proof relies on an explicit algorithm
in~\cite{DBMM2020} based on the seminal work of
Brouwer~\cite{Brouwer2010}. In Section~\ref{S3}, the role of the second largest
eigenvalue $\lambda_2$ of $\Gamma_{e_1,e_2}$ is elucidated in the expander
mixing lemma for bipartite graphs: we give a short proof
in Lemma~\ref{lem:eml}. In addition, we prove Proposition~\ref{lem:density}
which shows that bounds (lower and upper) can be determined simply by
computing two ratios $\alpha_1$ and $\alpha_2$. Bounds for the
orthogonal case are computed in Section~\ref{sec:GNP}, for the
symplectic and unitary cases in Section~\ref{sec:GNP2}, and computing
$\alpha_1,\alpha_2$ is key. The orthogonal case is hardest because of the
types of the (even dimensional) non-degenerate subspaces $S_1$ and $S_2$.
Finally, Section~\ref{S:future} discusses the general case $d>e_1+e_2$.

\section{Eigenvalues}\label{sec:eigenvalues}

The graph $\Gamma_{e_1,e_2}$ can be described in the spherical building
of type $\rA_{d-1}$, corresponding to the classical group $\textup{PSL}_d(q)$. 
Adjacency in $\Gamma_{e_1,e_2}$ 
corresponds to opposition in $\rA_{d-1}$ (that is, an $e_1$-space
and an $e_2$-space in $\rA_{d-1}$ are opposite in a 
building-theoretical sense precisely when they are complementary,
see~\cite{Brouwer2010} and~\cite[Lemma 3.7]{DBMM2020}).
The Coxeter diagram for $\rA_{d-1}$ is shown in Fig.~\ref{F}.
Brouwer observed in \cite[Theorem 1.1]{Brouwer2010} that for any
opposition graph of a spherical building over $\F_q$,
its eigenvalues are powers of~$q$.
Implicitly, \cite{Brouwer2010} describes an algorithm
to calculate the eigenvalues of graphs such as 
$\Gamma_{e_1,e_2}$.
This algorithm is explicitly stated
in \cite[Algorithms~1,\,2]{DBMM2020},
which we sometimes refer to simply as Algorithms~1,\,2.

\begin{figure}[!ht]
\centering
\begin{tikzpicture}[scale=1.4]  
  \node [above] at (0,0.1) {$1$};
  \node [above] at (1,0.1) {$2$};
  \node [above] at (2.5,0.1) {$e_2$};
  \node [above] at (4,0.1) {$e_1$};
  \node [above] at (5.5,0.1) {$d-2$};
  \node [above] at (6.5,0.1) {$d-1$};
  \draw (0,0)--(1.5,0);\draw (2,0)--(3,0);
  \draw (3.5,0)--(4.5,0);\draw (5,0)--(6.5,0);
  \draw [dotted, thick] (1.5,0) -- (2,0);
  \draw [dotted, thick] (3,0) -- (3.5,0);
  \draw [dotted, thick] (4.5,0) -- (5,0);
  \draw [fill] (0,0) circle [radius=0.1];
  \draw [fill] (1,0) circle [radius=0.1];
  \draw [fill] (2.5,0) circle [radius=0.1];
  \draw [fill] (4,0) circle [radius=0.1];
  \draw [fill] (5.5,0) circle [radius=0.1];
  \draw [fill] (6.5,0) circle [radius=0.1];
\end{tikzpicture}
\smallskip
\caption{The Coxeter diagram of $\rA_{d-1}$ where $d=e_1+e_2$ and $e_1\ge e_2$.}
\label{F}
\end{figure}
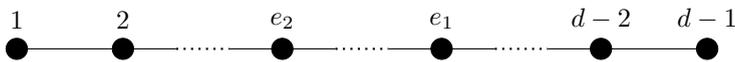

The key observation in \cite{Brouwer2010} is that 
we can calculate the eigenvalues of the oppositeness
relation from the irreducible characters of the Coxeter
group associated with the building.
In the case of $\rA_{d-1}$ this means
that we can calculate the eigenvalues of
any opposition graph in $\rA_{d-1}$ such as $\Gamma_{e_1,e_2}$
from the irreducible characters of the symmetric group $\Sym(d)$.

The symmetric group is viewed in this setting as a Coxeter group with
the set of adjacent transpositions $\{s_1,\dots,s_{d-1}\}$ as its set of 
generators $S$,
where $s_i = (i, i+1)$ for $i\in\{1,\dots,d-1\}$.

The unique longest word in $\Sym(d)$ with respect to the Coxeter generators
is denoted $w_0$. Its length is $\binom{d}{2}$ and
\[
 w_0=\begin{pmatrix}
   1 & 2 & \cdots & d \\
   d & d-1 & \cdots & 1
 \end{pmatrix}=\prod_{i=1}^{d-1}s_{d-1}\cdots s_{i+1}s_i.
\]
For instance, $w_0=s_3s_2s_1s_3s_2s_3$ when $d=4$.

We follow~\cite[Algorithms 1, 2]{DBMM2020}
to calculate the eigenvalues of $\Gamma_{e_1,e_2}$ \emph{up to sign}.
This suffices since the graph $\Gamma_{e_1,e_2}$ is bipartite, and
$\lambda$ is an eigenvalue  of $\Gamma_{e_1,e_2}$
if and only if $-\lambda$ is an eigenvalue.

\medskip 

We apply Algorithm 2 to the building of type $\rA_{d-1}$
with Coxeter group $W = \Sym(d)$.
An $e_1$-space is a partial flag of type $\{ e_1 \}$, its cotype is $J=
\{ 1, \ldots, d-1 \} \setminus \{ e_1 \}$.
Set $W_J = \langle s_i: i \in J \rangle$. Then
$W_J \cong \Sym(e_1) \times \Sym(e_2)$ as~$e_2=d-e_1$.

A \emph{partition} $\mu$ of $d$, denoted $\mu\vdash d$, is a
sequence $[\mu_1, \ldots, \mu_k]$ of positive integers with
$\mu_1\ge\cdots\ge\mu_k>0$ and $\sum_{i=1}^k\mu_i=d$. The irreducible
complex characters of $\Sym(d)$ have the form $\chi_\mu$ for a
unique $\mu\vdash d$. The parts of the conjugate partition $\mu^*$ of $\mu$
satisfy $\mu^*_i=|\{j\mid \mu_j\ge i\}|$.
We define
two invariants $a(\mu)$ and $a^*(\mu)$ as
\begin{equation}\label{E:aa*}
  a(\mu) = \sum_{i=1}^k (i-1) \mu_i,\quad\textup{and}\quad
  a^*(\mu) = \sum_{i=1}^k\frac{\mu_i(\mu_i-1)}{2}= \sum_{i=1}^k \binom{\mu_i}{2},
\end{equation}
and note that $a^*(\mu)=a(\mu^*)$, see \cite[\S\S 5.4.1, 5.4.2]{GP2000} and
\emph{c.f.}~\cite[Proposition~3.3]{DBMM2020}.

\begin{Proposition}[{\cite[Proposition 5.4.11]{GP2000}}]
  \label{prop:exponent}
  Let $d \geq 1$. Let $\chi_\mu$ denote a character of $\Sym(d)$
  corresponding to the partition $\mu$ of $d$.  Then 
  \[
  \binom{d}{2} \frac{\chi_\mu(r)}{\chi_\mu(1)} = a^*(\mu) - a(\mu),
  \quad\text{where $r\in\Sym(d)$ is a transposition.}
  \]
\end{Proposition}

Following \cite[Algorithm~1]{DBMM2020}, we denote by $R$ a set of 
representatives
of the conjugacy classes containing the generators in $S$. Then observe that $R$ 
comprises
one transposition $r$ since the conjugacy class $s_i^W$ comprises all
transpositions in $W=\Sym(d)$ for any $i \in \{1,\dots,d-1\}$,
so that $|r^W|=\binom{d}{2}$. Furthermore, the structure constant
$q_s$ in Algorithm~1 equals $q$ by the comment
after~\cite[Proposition~2.4]{DBMM2020}. In summary, the output of Algorithm~1
is the eigenvalue $\lambda_\mu$, where $\lambda^2_\mu=q^{e_\mu}$
and the value of
$e_\mu=\binom{d}{2}\left(1+\frac{\chi_\mu(s)}{\chi_\mu(1)}\right)$ is
independent of the choice of~$s\in S$.

Algorithm 2 applied to $W=\Sym(d)$ can be described as follows.
It is convenient to compute the eigenvalue $\lambda_\mu$ of 
$\chi_\mu$ up to sign, as remarked above.
\smallskip
\begin{enumerate}
\item Decompose the induced character $\mathrm{ind}^W_{W_J}(1_{W_J})$
  as a sum $\sum \chi_\mu$ of irreducible characters of $W$, and determine
  the relevant partitions $\mu$ of~$d$.
 \item For each $\mu$ appearing in Step 1, 
   calculate  using Proposition \ref{prop:exponent} the exponent $e_\mu = \binom{d}{2}\left(1+ \frac{\chi_\mu(r)}{\chi_\mu(1)}\right)$ where~$r$
   is a transposition.
 \item Calculate the length $\ell=\binom{e_1}{2}+\binom{e_2}{2}$ of
   the longest word in $W_J$, see below.
 \item The eigenvalues of $\Gamma_{e_1, e_2}$
   are now $\pm q^{e_\mu/2 - \ell}$ with $\mu$ determined in Step~1.
\end{enumerate}
\smallskip
This description concurs with that
of Algorithm 2 in \cite{DBMM2020}, except that in Step 2,
for the output of Algorithm 1 we use buildings of type $\rA_{d-1}$
and Proposition~\ref{prop:exponent}.

\begin{proof}[Proof of Proposition \ref{lem:ev}]
  Step 1 of Algorithm 2 determines, via Frobenius reciprocity,
  the irreducible characters of $W=\Sym(d)$
  that do not vanish when restricted to $W_J$.
  Precisely, we apply Pieri's rule~\cite[Corollary~6.1.7]{GP2000}
  to find the decomposition
  $\mathrm{ind}^W_{W_J}(1_{W_J}) = \sum_{j=0}^{e_1} \chi_{[d-j,j]}$.
  This completes Step 1 of Algorithm 2.
    
  For Step 2 of Algorithm 2, we apply Proposition~\ref{prop:exponent}
  to each character $\chi_{[d-j,j]}$ of $\Sym(d)$. Write
  $\mu = [ d-j, j ]$ where $d-j\ge e_1\ge e_2\ge j$.
  (When $j=0$, we identify $\mu_2=[d,0]$ with $\mu_2=[d]$.)
  The functions $a(\mu)$ and $a^*(\mu)$ in~\eqref{E:aa*} are:
  \begin{equation*}
    a(\mu) = j\qquad\textup{and}\qquad
    a^*(\mu) = \binom{d-j}{2} + \binom{j}{2}
    \quad\textup{for $0\le j\le e_2$.}
  \end{equation*}
  Hence, by Proposition \ref{prop:exponent},
  \[
  e_\mu=\binom{d}{2} + \binom{d-j}{2} + \binom{j}{2} - j
  = d^2 - d + j^2 - jd - j.
  \]
  This completes Step 2 of Algorithm 2.
    
  In Step 3, the length of the longest word $\ell$ in $W_J = \Sym(e_1) \times \Sym(e_2)$ is
  $\binom{e_1}{2} + \binom{e_2}{2}$.
  Thus $\ell = \frac{d^2-d-2e_1e_2}{2}$ and, by Step 4, the
  eigenvalue corresponding to $\chi_\mu$ is~$\pm q^{m_j}$~where
  \[
  m_j = \frac{e_\mu}{2} - \ell = \frac{j^2 - jd - j +2e_1e_2}{2}
  = e_1e_2 - \frac{j(d+1-j)}{2}. 
  \]
\end{proof}

\section{The Density Bound}\label{S3}

We state a version of the expander mixing lemma for regular,
bipartite graphs. This is stated in~\cite[Theorem~5.1]{Haemers1995} using the
language of block designs. Its proof is short,
so we include it here.
Given a graph $\Gamma$, we write $E(Y_1, Y_2)$ for the number of edges
between subsets $Y_1$ and $Y_2$ of the set of vertices of $\Gamma$.

\begin{Lemma}\label{lem:eml}
  Let $\Gamma$ be a $k$-regular bipartite graph with vertex set
  $X_1 \dcup X_2$. 
  Let $Y_i \subseteq X_i$ with $\alpha_i \coloneq |Y_i|/|X_i|$ for $i \in \{1,2\}$. 
  Suppose that the eigenvalues of the adjacency matrix of $\Gamma$ are 
  $\lambda_1\ge\lambda_2\ge\cdots\ge\lambda_{|X_1|+|X_2|}$.  Then
  \[
    \left|\frac{E(Y_1, Y_2)}{E(X_1, X_2)}-\alpha_1 \alpha_2 \right|
    \leq \frac{\lambda_2}{k} \sqrt{\alpha_1 \alpha_2 (1-\alpha_1) (1-\alpha_2)}.
  \]
\end{Lemma}

Since the graph $\Gamma$ has $|X_1|k = k|X_2|$ edges by $k$-regularity,
it follows that $|X_1| = |X_2|$.
Our proof follows \cite{Haemers1995} and uses {\it interlacing\,}:
Recall that 
the {\it quotient matrix} $B = (b_{ij})$ of a partition
$P_1 \dcup \cdots \dcup P_m$ of the vertex set of $\Gamma$ into $m$ parts
is an $(m\times m)$-matrix where $b_{ij} = E(P_i, P_j)/|P_i|$.
The eigenvalues of $B$ {\it interlace} those of the adjacency
matrix $A$ of $\Gamma$ by~\cite[Theorem~2.1]{Haemers1995},
that is if  the spectrum of $A$ is
$\lambda_1 \geq \lambda_2 \geq \cdots \geq \lambda_n$
and the spectrum of $B$ is $\mu_1 \geq \mu_2 \geq \cdots \geq \mu_m$,
then $\lambda_i \geq \mu_i \geq \lambda_{n-m+i}$ holds
for all $i \in \{ 1, \dots, m \}$.

\begin{proof}[Proof of Lemma~\ref{lem:eml}.]
  If $\alpha_1=0$, then $Y_1$ is empty. Hence $E(Y_1,Y_2)=0$ and the bound holds
  with equality. If $\alpha_1=1$, then $Y_1=X_1$ and by $k$-regularity
  \[
  \frac{E(Y_1,Y_2)}{E(X_1,X_2)}=\frac{k|Y_2|}{k|X_2|}=\alpha_2,
  \]
  and again the bound holds with equality. Similar arguments
  show that the bound  holds for $\alpha_2\in\{0,1\}$.
  Henceforth assume that $0<\alpha_1,\alpha_2<1$. 
  
  Write $D \coloneq E(X_1, X_2)$ and $E \coloneq E(Y_1, Y_2)$.
  Let $B=(b_{i,j})$ be the quotient matrix relative to the partition
  $Y_1 \dcup (X_1 \setminus Y_1) \dcup Y_2 \dcup (X_2 \setminus Y_2)$
  of the vertex set of~$\Gamma$. Then
  \begin{align*}
  b_{1,3}+b_{1,4}&=\frac{E(Y_1,X_2)}{|Y_1|}=k,\ 
  b_{1,3}=\frac{E(Y_1,Y_2)}{|Y_1|}=\frac{E}{\alpha_1|X_1|},\quad
  \textup{ and similarly}\\
  b_{2,3}+b_{2,4}&=\frac{E(X_1 \setminus Y_1,X_2)}{|X_{1} \setminus Y_{1}|}=k,\\ 
  b_{2,3}&=\frac{E(X_1 {\setminus} Y_1,Y_2)}{|X_{1} {\setminus} Y_{1}|}=
    \frac{E(X_1,Y_2){-}E}{(1-\alpha_1) |X_1|}=\frac{\alpha_2 D {-} E}{(1-\alpha_1) |X_1|}.
  \end{align*}
  Thus
  \begin{equation}\label{E:B}
      \renewcommand{\arraystretch}{1.4}
    B = \begin{pmatrix}
      0 & 0 & 
      \frac{E}{\alpha_1 |X_1|} 
      & k {-} \frac{E}{\alpha_1 |X_1|}\\
      0 & 0 & 
      \frac{\alpha_2 D - E}{(1-\alpha_1) |X_1|} &
      k {-} \frac{\alpha_2 D - E}{(1-\alpha_1) |X_1|} \\
      \frac{E}{\alpha_2 |X_2|} 
      & k {-} \frac{E}{\alpha_2 |X_2|} & 0 & 0 \\
      \frac{\alpha_1 D - E}{(1-\alpha_2) |X_2|} &
      k {-} \frac{\alpha_1 D - E}{(1-\alpha_2) |X_2|} & 0 & 0 
    \end{pmatrix}.
  \end{equation}

  Set $\delta=|X_1||X_2|\alpha_1\alpha_2(1-\alpha_1)(1-\alpha_2)$.
  Using a computer, we find that
  \[
  \det(tI - B) = (t^2-k^2)(t^2-\gamma^2/\delta)
  \qquad\textup{where $\gamma=E-D\alpha_1 \alpha_2$}.
  \]
  As $|X_1|=|X_2|$, the eigenvalues $\mu_1\ge\mu_2\ge\mu_3\ge\mu_4$
  of $B$  equal $\pm k,\pm\mu$  where
  \vskip-1mm
  \[
  \mu\coloneq\frac{\left|E-D\alpha_1 \alpha_2\right|}{ |X_1| \sqrt{\alpha_1 \alpha_2 (1-\alpha_1) (1-\alpha_2)}}.
  \]
  However, $\Gamma$ is $k$-regular, so its largest eigenvalue $\lambda_1$ is $k$
  by \cite[Proposition~1.3.8]{BH2012}. It follows from interlacing
  that $\mu_1=k$ and $\mu_2=\mu$.
  In addition, interlacing implies that $\mu_2\le\lambda_2$, that is
  \[
    \frac{\left|E(Y_1, Y_2) - E(X_1, X_2) \alpha_1 \alpha_2\right|}
   {|X_1| \sqrt{ \alpha_1 \alpha_2 (1-\alpha_1) (1-\alpha_2)}}
    \leq \lambda_2.
  \]
  Using $|X_1| = E(X_1, X_2)/k$ proves the assertion.
\end{proof}

\begin{proof}[Proof of Proposition~\ref{lem:density}]
  For $i\in\{1,2\}$, let $X_i$ denote the set of $e_i$-subspaces of~$V$.
  It follows from Proposition~\ref{lem:ev} that $\lambda_1=q^{e_1e_2}=k$
  and $\lambda_2=q^{e_1e_2-d/2}$.
Taking $\Gamma = \Gamma_{e_1,e_2}$ 
in Lemma~\ref{lem:eml} gives
\begin{align*}
 \frac{E(Y_1, Y_2)}{E(X_1, X_2)} \geq 
    \alpha_1 \alpha_2 - 
    \sqrt{\alpha_1\alpha_2 (1-\alpha_1)(1-\alpha_2)} q^{-\frac{d}{2}}.
\end{align*}
Since $E(X_1, X_2) = k \cdot |X_1|$, we have
\begin{align*}
 \frac{E(Y_1, Y_2)}{|Y_1| \cdot |Y_2|} 
 &= \frac{|X_1|^2}{|Y_1| \cdot |Y_2|} \cdot \frac{E(Y_1, Y_2)}{|X_1|^2}
 = (\alpha_1 \alpha_2)^{-1} \cdot \frac{k}{|X_1|} \cdot \frac{E(Y_1, Y_2)}{E(X_1, X_2)}\\
 &\geq \frac{k}{|X_1|} 
 \cdot \left(1 - \sqrt{\tfrac{(1-\alpha_1)(1-\alpha_2)}{\alpha_1\alpha_2}} q^{-\frac{d}{2}}\right)\\
 &=\frac{k}{|X_1|} 
 \cdot \left(1 - \sqrt{(\tfrac{1}{\alpha_1}-1)(\tfrac{1}{\alpha_2}-1)} q^{-\frac{d}{2}}\right).
\end{align*}
The first claimed inequality follows by rewriting $k/|X_1|$ because
\[
\frac{k}{|X_1|} =\frac{q^{e_1e_2}}{\gauss{d}{e_1}_q}
=\frac{\omega_q(e_1)\omega_q(e_2)}{\omega_q(e_1+e_2)}>\omega_q(e_2).
\]
The second inequality now follows from
$\sqrt{(\tfrac{1}{\alpha_1}-1)(\tfrac{1}{\alpha_2}-1)}\le \tfrac{1}{\alpha}-1$.
\end{proof}

\begin{Corollary}\label{lem:density2}
  The second bound in Proposition~$\ref{lem:density}$ implies that
  \smallskip
  \[
  \frac{|\{(S_1, S_2) \in Y_1 \times Y_2: S_1 \cap S_2 = 0\}|}{|Y_1| \cdot |Y_2|}  > \left(1-\frac{3}{2}q^{-1}\right)\left(1-\left(\frac{1}{\alpha}-1\right)q^{-d/2}\right).
  \]
\end{Corollary}

\begin{proof}
  Since $\omega_q(\infty)>1-q^{-1}-q^{-2}$ by \cite[Lemma~3.5]{NP1995},
  it follows that
  \[\omega_q(e_2)>\omega_q(\infty)>1-q^{-1}-q^{-2}\ge 1-\frac{3}{2q}.\]
  The result now follows from Proposition~\ref{lem:density}.
\end{proof}

\section{Proof in the orthogonal case} \label{sec:GNP}

In this section, we prove the orthogonal bound in Theorem~\ref{T:OSpU}.

\begin{Theorem}\label{lem:main}
  Suppose that $e_1,e_2 \geq 2$ are even and $V=(\F_q)^{e_1+e_2}$ is
  an $(e_1+e_2)$-dimensional vector space
  equipped with a non-degenerate quadratic form of type $\eps\in\{-,+\}$.
  For $i\in\{1,2\}$ and for $\sigma_i \in \{ -, + \}$,
  let $Y_i$ denote the set of all non-degenerate $e_i$-spaces of
  type $\sigma_i$.
  The proportion of pairs $(S_1, S_2) \in Y_1 \times Y_2$ for which
  $S_1\cap S_2=\{0\}$ is at least $1-\frac{3}{2q}$ for all $e_1,e_2\ge2$,
  $\eps,\sigma_1,\sigma_2\in\{-,+\}$ and all prime-powers $q\ge2$.
\end{Theorem}

\begin{proof}
  Let $e_1=2m_1$, $e_2=2m_2$ and let $V=(\F_q)^d$ be a non-degenerate
  orthogonal space of type $\eps\in\{-,+\}$ where $d=e_1+e_2$.
  We will assume, without loss of generality, that $e_2\leq e_1$
  and hence that $m_2\leq m_1$.
  Denote the isometry  group of $V$ by $\GO_d^\eps(q)$.
  We use the formula for $|\GO^\sigma_{2m}( q)|$ in ~\cite[p.\,141]{Taylor91},
  where $\sigma\in\{+,-\}$ is identified with $1,-1$, respectively. Since
  $q^{2m}-1=(q^m-\sigma)(q^m+\sigma)$, we have
  \[
  |\GO^\sigma_{2m}(q)| = 2 q^{m(m-1)}(q^m - \sigma) \prod_{i=1}^{m-1} (q^{2i}-1)
  =\frac{2 q^{m(2m-1)}}{1+\sigma q^{-m}}\prod_{i=1}^m(1-q^{-2i}).
  \]
  Hence $|\GO^\sigma_{2m}( q)|\sim 2q^{m(2m-1)}$ as $q\to\infty$.
  Recall that $\omega_{q^2}(m)=\prod_{i=1}^m(1-q^{-2i})$.
  
  Let $Y_1$ denote the set of non-degenerate $e_1$-spaces in $V$ of
  type $\sigma_1$. The stabilizer of $S_1\in Y_1$ in $\GO^{\eps}_{d}(q)$ is
  $\GO^{\sigma_1}_{e_1}(q) \times \GO^{\eps\sigma_1}_{e_2}(q)$
  since $S_1^\perp$ has type $\eps\sigma_1$
  by~\cite[Lemma 2.5.11(ii)]{KL}.
  It follows from the Orbit-Stabilizer Lemma that
  \begin{align*}
    |Y_1|=\frac{|\GO^\eps_d(q)|}{|\GO^{\sigma_1}_{e_1}(q){\times}\GO^{\eps\sigma_1}_{e_2}(q)|}
    &=\frac{q^{e_1e_2}(1{+}\sigma_1 q^{-m_1})(1{+}\eps\sigma_1 q^{-m_2})}{2(1+\eps q^{-m_1-m_2})}\frac{\omega_{q^2}(m_1+m_2)}{\omega_{q^2}(m_1)\omega_{q^2}(m_2)}.
  \end{align*}
  Hence $|Y_1|\sim\frac{1}{2}q^{e_1e_2}$ as $q\to\infty$.
  We shall write
  \[
  \gauss{d}{e_1}_q=\frac{q^{e_1e_2}}{B_{q}(e_1,e_2)}
  \quad\textup{where}\quad
  B_{q}(e_1,e_2)\coloneq\frac{\omega_q(e_1)\omega_q(e_2)}{\omega_q(e_1+e_2)}.
  \]
  Then
  \[
  |Y_1|=\frac{q^{e_1e_2}\lambda(\sigma_1,\eps)}{B_{q^2}(m_1,m_2)}
  \quad\textup{where}\quad \lambda(\sigma_1,\eps)
  =\frac{(1+\sigma_1 q^{-m_1})(1+\eps\sigma_1 q^{-m_2})}{2(1+\eps q^{-m_1-m_2})}.
  \]
  Hence
  \[
    \alpha_1=\frac{|Y_1|}{\gauss{d}{e_1}_q}
     =\frac{\lambda(\sigma_1,\eps)B_{q}(e_1,e_2)}{B_{q^2}(m_1,m_2)}.
  \]
  Proposition~\ref{lem:density} gives the lower bound
  \[
  \frac{E(Y_1,Y_2)}{|Y_1||Y_2|}
  \ge B_{q}(e_1,e_2)\left(1-\sqrt{\left(\frac{1}{\alpha_1}-1\right)\left(\frac{1}{\alpha_2}-1\right)}q^{-d/2}\right)
  \]
  A computer program~\cite{FI} checks that the above bound is greater
  than $1-\frac{1.5}{q}$ for all $q\le5$ and $1\le m_2\le m_1\le 6$
  except when $q\le5$ and $m_1=m_2=1$, or $(q,m_2,m_1)$
  equals $(2,1,2)$, $(2,1,3)$, $(2,2,2)$.
  For these seven exceptions, we used the  {\sc GAP} package
  {\sc FinInG}~\cite{FinIng} to do an exact count.
  The {\sc GAP/FinIng} code~\cite{FI} 
  verifies that the lower bound $1-\frac{1.5}{q}$ holds in these cases. 
  Thus when $q\le5$, we will henceforth assume that $m_1+m_2\ge7$,
  and hence that $d=2m_1+2m_2\ge14$.
  
  Observe now that $\lambda(\sigma_1,\eps)\ge\lambda(-,+)$.
  Take a lower bound $\alpha_1\ge\alpha$ and $\alpha_2\ge\alpha$ where
  $\alpha\coloneq\lambda(-,+)B_{q}(e_1,e_2)B_{q^2}(m_1,m_2)^{-1}$.
  Proposition~\ref{lem:density} gives the lower bound
  \begin{align}\label{E:bestLB}
  \frac{E(Y_1,Y_2)}{|Y_1||Y_2|}
  &\ge B_{q}(e_1,e_2)\left(1-\left(\frac{1}{\alpha}-1\right)q^{-d/2}\right)
  \nonumber\\
  &=B_{q}(e_1,e_2)\left(1+q^{-d/2}\right)-\lambda(-,+)^{-1}B_{q^2}(m_1,m_2)q^{-d/2}.
  \end{align}

  We next consider the case $q\le5$ and $m_1+m_2\ge7$. Since $B_{q^2}(m_1,m_2)$ is
  a decreasing function of $m_2$, it follows that
  \[
  B_{q^2}(m_1,m_2)\le B_{q^2}(7-m_2,m_2)\le B_{q^2}(1,6).
  \]
  Furthermore,   $B_q(e_1,e_2)\ge\omega_q(e_1)>\omega_q(\infty)$ and
  $1+q^{-d/2}>1$ so
  \[
  \frac{E(Y_1,Y_2)}{|Y_1||Y_2|}
  >\omega_q(\infty)-\lambda(-,+)^{-1}B_{q^2}(1,6) q^{-7}.
  \]
  Similar reasoning gives
  \begin{align*}
  \lambda(-,+)^{-1}=\frac{2(1+q^{-m_1-m_2})}{(1-q^{-m_1})(1-q^{-m_2})}
  &\le\frac{2(1+q^{-7})}{(1-q^{-m_1})(1-q^{-(7-m_1)})}\\
  &\le\frac{2(1+q^{-7})}{(1-q^{-1})(1-q^{-6})}.
  \end{align*}
  Using the more accurate lower bound $\omega_q(\infty)>1-q^{-1}-q^{-2}+q^{-5}$
  from \cite[Lemma~3.5]{NP1995}, one can check by computer that
  \[
    \frac{E(Y_1,Y_2)}{|Y_1||Y_2|}  >1-q^{-1}-q^{-2}+q^{-5}-\frac{2(1+q^{-7})}{(1-q^{-1})(1-q^{-6})}B_{q^2}(1,6)q^{-7}>1-\frac{3}{2q}
  \]
  holds for $q\le5$.
  
  Finally, suppose that $q\ge7$ and $1\le m_2\le m_1$ holds. When
  $m_1=m_2=1$, it follows from~\eqref{E:bestLB} and $1+q^{-d/2}>1$ that
  \[
  \frac{E(Y_1,Y_2)}{|Y_1||Y_2|}
  >\frac{1}{(1+q^{-1}+q^{-2})(1+q^{-2})}-\frac{2q^{-2}}{(1-q^{-1})^2}.
  \]
  This is bound is dominated by the first term and is a
  decreasing function of $q$. Hence the bound is
  greater than $1-\frac{1.5}{q}$ for all $q\ge7$.
  It remains to consider
  the case $m_1+m_2\ge3$ and hence $d\ge6$.
  Arguing as above, and using $q\ge7$, gives
  \[
  \frac{E(Y_1,Y_2)}{|Y_1||Y_2|}
  >1-q^{-1}-q^{-2}+q^{-5}-\frac{2(1+q^{-3})}{(1{-}q^{-1})(1{-}q^{-2})}B_{q^2}(1,2)q^{-3}>1-\frac{3}{2q}.
  \]
  Thus in all cases the bound $1-\frac{3}{2q}$ holds, as claimed.
\end{proof}

\section{Proof in the symplectic and unitary cases} \label{sec:GNP2}

In this section, we prove the symplectic and unitary bounds
in Theorem~\ref{T:OSpU}.

\begin{Theorem}\label{T:Sp}
  Suppose that $e_1,e_2 \geq 2$ are even and $V=(\F_q)^{e_1+e_2}$ is
  an $(e_1+e_2)$-dimensional symplectic space.
  For $i\in\{1,2\}$ let $Y_i$ denote the set of all non-degenerate
  $e_i$-spaces of $V$.
  The proportion of pairs $(S_1, S_2) \in Y_1 \times Y_2$ for which
  $S_1\cap S_2=\{0\}$ is at least $1-\frac{10}{7q}$ for all $e_1,e_2$ and all prime-powers $q\ge2$.
\end{Theorem}

\begin{proof}
  Let $e_1=2m_1$, $e_2=2m_2$ and let $V=(\F_q)^d$ be a non-degenerate
  symplectic space where $d=e_1+e_2$. As before, we shall assume, without loss of generality,
  that $2\le e_2\le e_1$ and hence that $1\le m_2\le m_1$. Let $m=m_1+m_2$. The isometry
  group $\Sp_{2m}(q)$ of $V$ has order
  $q^{m^2}\prod_{i=1}^{m} (q^{2i}-1)=q^{2m^2+m}\omega_{q^2}(m)$
  by~\cite[p.\,70]{Taylor91}.
  
  Let $Y_i$ denote the set of non-degenerate $e_i$-spaces in $V$.
  Clearly $|Y_1|=|Y_2|$.
  The stabilizer of $S_1\in Y_1$ in $\Sp(V)$ is
  $\Sp(S_1) \times \Sp(S_1^\perp)$.  Therefore
  \begin{align*}
    |Y_2|=|Y_1|=\frac{|\Sp_{e_1+e_2}(q)|}{|\Sp_{e_1}(q){\times}\Sp_{e_2}(q)|}
    &=\frac{q^{e_1e_2}\omega_{q^2}(m_1+m_2)}{\omega_{q^2}(m_1)\omega_{q^2}(m_2)}\\
    &=\frac{q^{e_1e_2}}{B_{q^2}(m_1,m_2)}\quad
    \text{where $B_q(e_1,e_2)=\frac{\omega_q(e_1)\omega_q(e_2)}{\omega_q(e_1+e_2)}$.}
  \end{align*}
  Since
  $|X_1|=\gauss{d}{e_1}_q=\gauss{d}{e_2}_q=|X_2|=q^{e_1e_2}/B_{q}(e_1,e_2)$,
  we have
  \[
  \alpha_1=\frac{|Y_1|}{|X_1|}
  =\frac{q^{e_1e_2}}{B_{q^2}(m_1,m_2)}\cdot\frac{B_{q}(e_1,e_2)}{q^{e_1e_2}}
  =\frac{B_q(e_1,e_2)}{B_{q^2}(m_1,m_2)}.
  \]
  Proposition~\ref{lem:density} gives the lower bound
  \begin{align*}\label{E:bestLB}
    \frac{E(Y_1,Y_2)}{|Y_1||Y_2|}
    &\ge  B_q(e_1,e_2)\left(1-\left(\frac{1}{\alpha_1}-1\right)q^{-d/2}\right)\\
    &\ge  B_q(e_1,e_2) -\left(B_{q^2}(m_1,m_2)-B_q(e_1,e_2)\right)q^{-d/2}\\
    &>  B_q(e_1,e_2) - B_{q^2}(m_1,m_2)q^{-d/2}.
  \end{align*}

  Note that $B_q(e_1,e_2)\ge\omega_q(e_1)>\omega_q(\infty)>1-q^{-1}-q^{-2}$
  by~\cite[Lemma~3.5]{NP1995}. Also, $1>B_{q^2}(m_1,m_2)$ and $d=e_1+e_2\ge4$.
  If $q\ge5$, we have
  \[
  \frac{E(Y_1,Y_2)}{|Y_1||Y_2|}\ge1-q^{-1}-q^{-2}-q^{-2}> 1-\frac{10}{7}q^{-1}.
  \]
  For $q\in\{2,3,4\}$ we have $\omega_2(\infty)>0.288$,
  $\omega_3(\infty)>0.56$ and $\omega_4(\infty)>0.688$. Thus
  when $q\in\{2,3,4\}$ and $d\ge20$, we have
  \[
  \frac{E(Y_1,Y_2)}{|Y_1||Y_2|}\ge\omega_q(\infty)-q^{-10}  >1-\frac{10}{7}q^{-1}.
  \]
  Finally, if $q\in\{2,3,4\}$ and $d=e_1+e_2<20$, then a computer
  program shows that the last inequality below is satisfied
  \[
  \frac{E(Y_1,Y_2)}{|Y_1||Y_2|}\ge B_q(e_1,e_2)(1+q^{-d/2}) - B_{q^2}(m_1,m_2)q^{-d/2}>1-\frac{10}{7}q^{-1}.
  \]
\end{proof}

\begin{Theorem}\label{T:U}
  Suppose $V=(\F_{q^2})^{e_1+e_2}$ is
  an $(e_1+e_2)$-dimensional hermitian space where $e_1,e_2 \geq 1$.
  For $i\in\{1,2\}$, let $Y_i$ denote the set of all non-degenerate
  $e_i$-spaces of $V$.
  The proportion of pairs $(S_1, S_2) \in Y_1 \times Y_2$ for which
  $S_1\cap S_2=\{0\}$ is at least $1-\frac{c}{q^2}$ where
  $c=2$ when $(e_1,e_2,q)=(1,1,2)$,
  $c=\frac{3}{2}$ when $\min\{e_1,e_2\}=1$ and $(e_1,e_2,q)\ne(1,1,2)$,
  and $c=1.26$ otherwise.
\end{Theorem}

\begin{proof}
  Let $V=(\F_{q^2})^d$ be a non-degenerate unitary space where $d=e_1+e_2$.
  Let $\GU_d(q)$ denote the isometry group of $V$. We have
  $|\GU_d(q)|\sim q^{d^2}$ as $q\to\infty$,  more precisely
  $|\GU_d(q)|= q^{d(d-1)/2}\prod_{i=1}^{d} (q^i-(-1)^i)=q^{d^2}\omega_{-q}(d)$
  by~\cite[p.\,118]{Taylor91}~where
  \[
    \omega_{-q}(d)=\prod_{i=1}^d(1-(-q)^{-i}).
  \]

  The proportion of pairs $(S_1, S_2) \in Y_1 \times Y_2$ for which
  $S_1\cap S_2=\{0\}$ is unchanged if we swap the subscripts. Thus we can
  henceforth assume that $1\le e_2\le e_1$. Proposition~\ref{lem:density}
  gives poor bounds when $e_2=1$. When $e_2=1$, it turns out to be simple to
  overestimate the complementary proportion using ideas in the proof
  of~\cite[Theorem~4.1]{GNP2022}. The proportion of non-degenerate
  $e_1$-subspaces (where $e_1=d-1$) that contain a given 1-subspace is
  at most $c_1/q^2$ where $c_1\coloneq\frac{1+q^{-e_1}}{1-q^{-1-e_1}}$
  by~\cite[p.\,9]{GNP2022}. This equals $\frac{2}{q^2}$
  if $(e_1,e_2,q)=(1,1,2)$. If $e_1\ge2$, we have
  $c_1\le\frac{1+2^{-2}}{1-2^{-3}}<\frac{3}{2}$, and if $q\ge3$ we have
  $c_1\le\frac{1+q^{-1}}{1-q^{-2}}=\frac{1}{1-q^{-1}}\le\frac{3}{2}$.
  Thus when $e_1 \ge 2$ and $e_2=1$ the complementary proportion is
  at most $\frac{3}{2q^2}$. This proves the claim
  when $e_2=1$.  We henceforth assume that $2\le e_2\le e_1$.
  
  The stabilizer of $S_1\in Y_1$ in $\GU(V)$ is 
  $\GU(S_1) \times \GU(S_1^\perp)$.  Hence
  \begin{align*}
    |Y_1|=|Y_2|=\frac{|\GU_{e_1+e_2}(q)|}{|\GU_{e_1}(q){\times}\GU_{e_2}(q)|}
    &=\frac{q^{2e_1e_2}\omega_{-q}(e_1+e_2)}{\omega_{-q}(e_1)\omega_{-q}(e_2)}
    =\frac{q^{2e_1e_2}}{B_{-q}(e_1,e_2)}.
  \end{align*}
  Replacing $q$ with $q^2$ shows that $|X_1|=|X_2|=(q^2)^{e_1e_2}/B_{q^2}(e_1,e_2)$,
  and hence
  \[
  \alpha_1=\frac{|Y_1|}{|X_1|}  =\frac{B_{q^2}(e_1,e_2)}{B_{-q}(e_1,e_2)}.
  \]
  Proposition~\ref{lem:density}, with $q$ replaced with $q^2$,
  gives the lower bound
  \begin{align}
    \frac{E(Y_1,Y_2)}{|Y_1||Y_2|}
    &\ge B_{q^2}(e_1,e_2)
    \left(1-\left(\frac{1}{\alpha_1}-1\right)q^{-d}\right)\nonumber \\
    &=B_{q^2}(e_1,e_2)(1+q^{-d}) -  B_{-q}(e_1,e_2)q^{-d}\nonumber \\
    &> B_{q^2}(e_1,e_2)-B_{-q}(e_1,e_2)q^{-d}.\label{EX}
   \end{align}
  Note that
  $B_{q^2}(e_1,e_2)\ge\omega_{q^2}(e_1)>\omega_{q^2}(\infty)>1-q^{-2}-q^{-4}$
  by~\cite[Lemma~3.5]{NP1995}.
    
  To find an upper bound for
  $B_{-q}(e_1,e_2)=\prod_{i=1}^{e_2}\frac{(1-(-q)^{-i})}{(1-(-q)^{-e_1-i})}$,
  we will use
  \begin{equation}\label{EE}
    (1-q^{-2i})(1+q^{-(2i+1)})<1<(1+q^{-(2j-1)})(1-q^{-2j}).
  \end{equation}
  It follows from~\eqref{EE} that
  $\prod_{i=1}^{e_2}(1-(-q)^{-i})\le 1+q^{-1}$ for all $e_2\ge1$ and
  \[
  \prod_{i=1}^{e_2}(1-(-q)^{-e_1-i}) \ge
  \begin{cases}
    1&\textup{if $e_1$ even, $e_2$ even,}\\
    1+q^{-e_1-e_2}&\textup{if $e_1$ even, $e_2$ odd,}\\
    1-q^{-e_1-1}&\textup{if $e_1$ odd, $e_2$ even,}\\
    (1-q^{-1-e_1})(1-q^{-e_1-e_2})&\textup{if $e_1$ odd, $e_2$ odd.}
  \end{cases}
  \]
  Hence the above product is greater than or equal
  to $(1-q^{-4})(1-q^{-6})$ for all $2\le e_2\le e_1$. Therefore
  $B_{-q}(e_1,e_2)\le\frac{1+q^{-1}}{(1-q^{-4})(1-q^{-6})}$.

  Since $d\ge4$, it follows from~\eqref{EX} that
  \[
    \frac{E(Y_1,Y_2)}{|Y_1||Y_2|}
    >1-q^{-2}-q^{-4}-    \frac{(1+q^{-1})q^{-4}}{(1-q^{-4})(1-q^{-6})}.
  \]
  This is greater than $1-\frac{1.26}{q^2}$ for all $q\ge4$.

  Suppose now that $q\in\{2,3\}$ and $d\ge10$. Then
  \[
   \frac{E(Y_1,Y_2)}{|Y_1||Y_2|}
   >\omega_{q^2}(\infty) - \frac{(1+q^{-1})q^{-10}}{(1-q^{-4})(1-q^{-6})}
  \]
  and since $\omega_{4}(\infty)>0.6885$, $\omega_{9}(\infty)>0.876$, the
  above bound is greater than $1-\frac{1.26}{q^2}$ for $q\in\{2,3\}$.
  It remains to consider $q\le 3$ and $2\le e_2\le e_1$ where $e_1+e_2<10$.
  In this case, a simple computer program verifies that
  \[
  \frac{E(Y_1,Y_2)}{|Y_1||Y_2|}
  \ge B_{q^2}(e_1,e_2)(1+q^{-d})-B_{-q}(e_1,e_2)q^{-d}> 1-\frac{1.26}{q^2}.
  \]
\end{proof}

\section{Future Work}\label{S:future}

A more general problem when $e_1+e_2<d$ is considered  in \cite{GNP2022a}.
Here $V=\F^d$ is a finite non-degenerate classical space and
a lower bound of the form $1-\frac{c}{|\F|}$ is sought for the proportion
of non-degenerate pairs $(S_1,S_2)$ satisfying $\dim(S_i)=e_i$ and
$S_1\cap S_2=\{0\}$. The bound given in~\cite[Theorem~1.1]{GNP2022a}
has the form $1-\frac{c}{|\F|}$ in the symplectic case, and
the orthogonal case for $q>2$, but only $1-\frac{c}{|\F|^{1/2}}$ in the
unitary case.
If one could compute the second eigenvalue $\lambda_2$ 
of the bipartite graph $\Gamma_{d,e_1,e_2}$ when $e_1+e_2<d$,
\emph{c.f.} Proposition~\ref{lem:ev}, then
it may be possible to obtain sharper lower bounds of the form $1-\frac{c}{|\F|}$
via Lemma~\ref{lem:eml}, in all cases.

\smallskip

 \paragraph*{Acknowledgment} 
 We thank the referee for their very helpful comments.
The first author is supported by the Australian Research Council
Discovery Grant DP190100450.
The second author is supported by a postdoctoral
fellowship of the Research Foundation -- Flanders~(FWO).

\vskip3mm
Data sharing is not applicable to this article as no datasets were generated or analysed during the current study.

\end{document}